\documentclass[11pt]{amsart}
\usepackage{amsmath,amsthm}
\usepackage{amstext,amsfonts,bm,amssymb}
\usepackage{graphics,graphicx} 
\usepackage{a4wide}
\usepackage{float}
\usepackage{color}
\usepackage{verbatim}

\usepackage{gastex}

\theoremstyle{plain}
\newtheorem{theorem}{Theorem}
\newtheorem{Prop}{Proposition}

\newtheorem{Lem}{Lemma}

\def\fl#1{\left\lfloor#1\right\rfloor}

\def\sttf2#1#2{\left[\!\!\left[#1\atop#2\right]\!\!\right]}  
\def\stf3f#1#2{\left[\!\!\left[\!\!\left[#1\atop#2\right]\!\!\right]\!\!\right]} 
\def\stff4#1#2{\left[\!\!\left[\!\!\left[\!\!\left[#1\atop#2\right]\!\!\right]\!\!\right]\!\!\right]}
\def\stss2#1#2{\left\{\!\!\left\{#1\atop#2\right\}\!\!\right\}}

\allowdisplaybreaks 

\begin{document}

\hbadness 5000
\vbadness 10000

\hfuzz 5pt

\title[A $q$-multiple zeta function]{Some explicit values of a $q$-multiple zeta function whose denominator power is not uniform} 

\author{Yuri Bilu}
\address{Institut de Math\'ematiques de Bordeaux UMR 5251\\ Universit\'e de Bordeaux and CNRS \\Talence 33405 France}
\email{yuri@math.u-bordeaux.fr} 

\author{Hideaki Ishikawa}
\address{School of Education \\University of Toyama \\Toyama 930-8555\\ Japan}
\email{ishikawa@edu.u-toyama.ac.jp} 

\author{Takao Komatsu}
\address{Institute of Mathematics\\ Henan Academy of Sciences\\ Zhengzhou 450046\\ China;  \linebreak
Department of Mathematics, Institute of Science Tokyo, 2-12-1 Ookayama, Meguro-ku, Tokyo 152-8551, Japan}
\email{komatsu.t.al@m.titech.ac.jp;\,komatsu@zstu.edu.cn}
\thanks{T.K. is the corresponding author.}

\date{%\today
}

\begin{abstract}
One of the generalizations of multiple zeta values is the $q$-version, and in the case of finite sums, they may be expressed explicitly in polynomial form. Several results have been found when the powers of the factors in the denominator are equal and when they are small. In this paper, we give explicit formulas for the case when the powers are unequal and are small.  
\medskip

\end{abstract}


\maketitle
%%%%%%%%%%%%%%%%%%%%%%%%%%%%%%%%%%%%%%%%%%%%%%%%%%%%%%%%%%%%%%%%%%%%%%%%%%%%%%%

\section{Introduction}\label{sec:1} 

For positive integers $s_1,s_2,\dots,s_m$, the multiple zeta function of the form 
\begin{equation}
\zeta(s_1,s_2,\dots,s_m):=\sum_{1\le i_1<i_2<\dots<i_m}\frac{1}{i_1^{s_1}i_2^{s_2}\dots i_m^{s_m}}
\label{def:mzv}
\end{equation}
and its generalizations or modifications have been studied by many researchers. It is believed that Euler initiated such a study, in the special case  $m=2$, but there had been little progress for more than 200 years until the brilliant works by Hoffman, Zagier, Zhang, and others in the early 1990s. See the book \cite{Zhao_book}.  
%The simple case is dealt with as $s_1=s_2=\dots=s_m$, in particular, the case $s_1=s_2=\dots=s_m=1$ is simplest.  

For the infinite sum to converge, we need to have $s_m\ge 2$. The simplest case is $s_1=s_2=\dots=s_{m-1}=1$ and $s=s_m$ is arbitrary.  Following \cite{Young1,Young2}, $\zeta(\underbrace{1,\dots,1}_{m-1},s)$ is called the {\it height $1$ multiple zeta function of depth $m$}. This height $1$ multiple zeta function has strong connections with Stirling, poly-Bernoulli, harmonic, hyperharmonic, and Roman harmonic numbers, and the multiple Hurwitz zeta function of order $r$ \cite{Ruijsenaars,Young14}. 

Finite multiple zeta functions, or  {\it multiple harmonic sums}, are defined as 
$$
\zeta_n(s_1,s_2,\dots,s_m):=\sum_{1\le i_1<i_2<\dots<i_m\le n}\frac{1}{i_1^{s_1}i_2^{s_2}\dots i_m^{s_m}}\,,
$$  
so that $\zeta(s_1,s_2,\dots,s_m)=\lim_{n\to\infty}\zeta_n(s_1,s_2,\dots,s_m)$. 

A different type of generalization is in terms of the \textit{ $q$-numbers }
$$
[n]_q:=\frac{1-q^n}{1-q}\quad(q\ne 1)\,. 
$$ 
Then the usual integers are recovered as $n=\lim_{q\to 1}[n]_q$. Various types of the $q$-multiple zeta functions (e.g., \cite{Bradley,OOZ,Zhao,Zudilin}) and the finite $q$-multiple zeta functions (e.g., \cite{BTT18,BTT20,Takeyama09,Tasaka21}) are introduced and studied. 

In this paper, we consider the function 
\begin{equation}
\mathfrak Z_n(q;;s_1,s_2,\dots,s_m):=\sum_{1\le i_1<i_2< \dots<i_m\le n-1}\frac{1}{(1-q^{i_1})^{s_1}(1-q^{i_2})^{s_2}\cdots(1-q^{i_m})^{s_m}}\,. 
\label{def:qssmzv}
\end{equation}
When $n\to\infty$, the  infinite version was studied by Schlesinger \cite{Schlesinger}:  
\begin{equation}
\mathfrak Z(q;;s_1,s_2,\dots,s_m):=\sum_{1\le i_1<i_2< \dots<i_m}\frac{1}{(1-q^{i_1})^{s_1}(1-q^{i_2})^{s_2}\cdots(1-q^{i_m})^{s_m}}
\label{def:qsmzv}
\end{equation}  
Note that in (\ref{def:qssmzv}) the upper bound is not $n$ but $n-1$. Nevertheless, we use the notation $\mathfrak Z_n$ because we take $q=\zeta_n:=\exp(2\pi\sqrt{-1}/n)$ as the special value in this paper.  

In the infinite case like (\ref{def:qsmzv}), the values of the function can be often expressed in terms of other elementary functions, similarly to the values of the function defined in (\ref{def:mzv}).  
In the finite case  like (\ref{def:qssmzv}), the values of the function can be expressed explicitly by choosing suitable $q$ and $s_1,s_2,\dots,s_m$. In our previous works (\cite{CK,Ko25b,Ko25a,KL,KP,KW}), we have considered the case where $s_1=s_2=\dots=s_m$, then specialize $q=\zeta_n$ or similarly. 
As noted in \cite{BTT20}, no explicit expression has been known unless $s_1=s_2=\dots=s_m$. 

In this paper, we deal with the finite $q$-multiple zeta values in the case where  one of $s_1, \ldots, s_m$ is $2$, and all the others  are $1$; for example, $s_1=s_2=\dots=s_{m-1}=1$ and $s_m=2$.

It looks hard to determine the imaginary part of  $\mathfrak Z_n(\zeta_n;;\underbrace{1,\dots,1}_{m-1},2)$. However, using a certain symmetric relation, we can find out the real part explicitly. The formula (see Theorem \ref{th:sym12111} below) is surprisingly neat, but proving it is more troublesome than one might imagine, because 
we can no longer use the properties of Stirling numbers or directly apply elementary symmetric functions, as we did in our previous work.

\section{More variations about the sum of symmetric values} 

In a previous paper \cite{Ko25a}, it is shown that 
$$
\mathfrak Z_n(\zeta_n;;\underbrace{1,\dots,1}_m)=\frac{1}{m+1}\binom{n-1}{m}
$$ 
and 
\begin{equation}
\mathfrak Z_n(\zeta_n;;s)=\left|\begin{array}{ccccc}
\frac{n-1}{2}&1&0&\cdots&\\ 
\frac{2}{3}\binom{n-1}{2}&\frac{n-1}{2}&1&&\vdots\\ 
\vdots&&\ddots&&0\\
\frac{s-1}{s}\binom{n-1}{s-1}&\frac{1}{s-1}\binom{n-1}{s-2}&\cdots&\frac{n-1}{2}&1\\ 
\frac{s}{s+1}\binom{n-1}{s}&\frac{1}{s}\binom{n-1}{s-1}&\cdots&\frac{1}{3}\binom{n-1}{2}&\frac{n-1}{2}\\ 
\end{array}
\right|\,.
\label{eq:zz-det2}
\end{equation}
In particular, by taking $s=2,3,\dots,9$ in (\ref{eq:zz-det2}), we have 
{\small
\begin{align*}
\mathfrak Z_n(\zeta_n;;2)&=-\frac{(n-1)(n-5)}{12}\,,\\
\mathfrak Z_n(\zeta_n;;3)&=-\frac{(n-1)(n-3)}{8}\,,\\
\mathfrak Z_n(\zeta_n;;4)&=\frac{(n-1)(n^3+n^2-109 n+251)}{6!}\,,\\
\mathfrak Z_n(\zeta_n;;5)&=\frac{(n-1)(n-5)(n^2+6 n-19)}{288}\,,\\
\mathfrak Z_n(\zeta_n;;6)&=-\frac{(n-1)(2 n^5+2 n^4-355 n^3-355 n^2+11153 n-19087)}{12\cdot 7!}\,,\\
\mathfrak Z_n(\zeta_n;;7)&=-\frac{(n-1)(n-7)(2 n^4+16 n^3-33 n^2-376 n+751)}{24\cdot 6!}\,,\\
\mathfrak Z_n(\zeta_n;;8)&=\frac{(n-1)(3 n^7+3 n^6-917 n^5-917 n^4+39697 n^3+39697 n^2-744383 n+1070017)}{10!}\,,\\
\mathfrak Z_n(\zeta_n;;9)&=\frac{27(n-1)(n-3)(n-9)(n^5+13 n^4+10 n^3-350 n^2-851 n+2857)}{2\cdot 10!}\,.
\end{align*} 
}%

No explicit expression has yet been found about the case where not all values of $s=j$ ($j=1,2,\dots,m$) are equal, although such cases may be curious. The case where only one value $1$ is replaced by $2$ was hard enough. See \cite{BTT20} for related conjectures.   
For non-negative integers $a$ and $b$, we prove the following.  

\begin{theorem} 
Let $a$ and $b$ be non-negative integers such that ${m:=a+b+1\ge2}$. Then for ${n\ge 2}$ we have
%For integers $n$ and $m$ with $n,m\ge 2$,  
\begin{equation}
\mathfrak Z_n(\zeta_n;;\underbrace{\underbrace{1,\dots,1}_a,2,\underbrace{1,\dots,1}_b}_m)+\mathfrak Z_n(\zeta_n;;\underbrace{\underbrace{1,\dots,1}_b,2,\underbrace{1,\dots,1}_a}_m)=-\frac{m!(n-2 m-3)}{(m+2)!}\binom{n-1}{m}\,.
\label{eq:sym12111}
\end{equation} 
\label{th:sym12111}
\end{theorem}

Furthermore, for non-negative integers $a$, $b$, $a'$ and $b'$ such that 
${m=a+b+1=a'+b'+1}$,  
we can show that 
{\small
$$
\mathfrak{Re}\left(\mathfrak Z_n(\zeta_n;;\underbrace{\underbrace{1,\dots,1}_a,2,\underbrace{1,\dots,1}_b}_m)+\mathfrak Z_n(\zeta_n;;\underbrace{\underbrace{1,\dots,1}_{b'},2,\underbrace{1,\dots,1}_{a'}}_m)\right)=-\frac{m!(n-2 m-3)}{(m+2)!}\binom{n-1}{m}\,.
$$ 
}%
This means that the real part of $\mathfrak Z_n(\zeta_n;;\underbrace{\underbrace{1,\dots,1}_a,2,\underbrace{1,\dots,1}_b}_m)$ depends only on~$m$, and is the same for any choice of $a$ and $b$ (though the imaginary part may vary):  
$$
\mathfrak{Re}\left(\mathfrak Z_n(\zeta_n;;\underbrace{\underbrace{1,\dots,1}_a,2,\underbrace{1,\dots,1}_b}_m)\right)=-\frac{m!(n-2 m-3)}{2(m+2)!}\binom{n-1}{m}. 
$$ 
This follows from \eqref{eq:sym12111} and the following lemma.

\begin{Lem}  
For any positive integers $s_1, \ldots, s_m$ we have 
$$
\mathfrak{Im}\left(\mathfrak Z_n(\zeta_n;;s_1,s_2,\dots,s_m)+\mathfrak Z_n(\zeta_n;;s_m,\dots,s_2,s_1)\right)=0\,. 
$$ 
\label{lem:symmetric}
\end{Lem}
\begin{proof}
We have 
\begin{align*}
\overline{\mathfrak Z_n(\zeta_n;;s_1,s_2,\dots,s_m)}
&=\sum_{1\le i_1< i_2< \dots< i_m\le n-1}\frac{1}{(1-\overline{\zeta_n}^{i_1})^{s_1}(1-\overline{\zeta_n}^{i_2})^{s_2}\cdots(1-\overline{\zeta_n}^{i_m})^{s_m}}\\
&=\sum_{n-1\ge n-i_1> n-i_2> \dots> n-i_m\ge 1}\frac{1}{(1-\zeta_n^{n-i_1})^{s_1}(1-\zeta_n^{n-i_2})^{s_2}\cdots(1-\zeta_n^{n-i_m})^{s_m}}\\
&=\mathfrak Z_n(\zeta_n;;s_m,s_{m-1},\dots,s_1)\,. 
\end{align*}
\end{proof}

The proof of Theorem \ref{th:sym12111} relies on the following lemma, which might be of independent interest.  For convenience, put 
\begin{equation}
u_r:=\frac{1}{1-\zeta_n^r}\quad(1\le r\le n-1)\,.
\label{eq:ur}
\end{equation}

\begin{Lem}
The real part of 
$$
\mathfrak Z_n(\zeta_n;;\underbrace{1,\dots,1}_{j-1},2,\underbrace{1,\dots,1}_{m-j})=\sum_{1\le i_1<\dots<i_m\le n-1}u_{i_1}\cdots u_{i_{j-1}}u_{i_j}^2 u_{i_{j+1}}\cdots u_{i_m}
$$ 
is a polynomial in $n$ of degree $m+1$ with rational coefficients. 
\label{lem:realsym12111}
\end{Lem} 

\begin{proof} 
We denote
$$
P_m^{(j)}(n):=\mathfrak Z_n(\zeta_n;;\underbrace{1,\dots,1}_{j-1},2,\underbrace{1,\dots,1}_{m-j})=\sum_{1\le i_1<\dots<i_m\le n-1}u_{i_1}\cdots u_{i_{j-1}}u_{i_j}^2 u_{i_{j+1}}\cdots u_{i_m}. 
$$ 
%though the result is independent from $j$, in fact.  
Notice that 
$$
u_r=\frac{1}{2}+\frac{\sqrt{-1}}{2}\cot\frac{r\pi}{n}
$$
and 
$$
u_r^2=\frac{1}{4}\left(1-\cot^2\frac{r\pi}{n}-2\sqrt{-1}\cot\frac{r\pi}{n}\right)\,.
$$
Consider the average of $P_m^{(j)}(n)$: 
$$
Q_m(n)=\frac{1}{m}\sum_{j=1}^m P_m^{(j)}(n)\,. 
$$ 

By using the elementary symmetric function, we write 
$$
e_k:=e_k(u_1,\dots,u_{n-1})=\sum_{1\le i_1<\dots<i_k\le n-1}u_{i_1}\cdots u_{i_k}\,. 
$$ 
Since 
\begin{align*}
e_1 e_m&=\sum_{j=1}^{n-1}u_j\sum_{1\le i_1<\dots<i_m\le n-1}u_{i_1}\cdots u_{i_m}\\
&=(m+1)e_{m+1}+\sum_{|S|=m}\left(\sum_{j\in S}u_j\right)\prod_{h\in S}u_h\\
&=(m+1)e_{m+1}+\sum_{j=1}^m P_m^{(j)}(n)\\
&=(m+1)e_{m+1}+m Q_m(n)\,,  
\end{align*}
we have 
\begin{equation}
m Q_m(n)=e_1 e_m-(m+1)e_{m+1}\,.
\label{eq:qee}
\end{equation} 
Here, for a given ($m+1$)-element subset $H$, the product $\prod_{h\in H}u_h$ appears in $e_1 e_m$ exactly $m+1$ times (choose any element of $H$ as $u_j$, the rest as the $m$-tuple). 
By Lemma \ref{lem:symmetric}, $P_m^{(j)}(n)=\overline{P_m^{(m+1-j)}(n)}$. Hence, $\mathfrak{Re}(P_m^{(j)}(n))=\mathfrak{Re}\bigl(P_m^{(m+1-j)}(n)\bigr)$. 
By symmetry of the symmetric group $S_m$ acting on positions, all $\mathfrak{Re}\bigl(P_m^{(j)}(n)\bigr)$ are equal. Therefore, 
$$
Q_m(n)=\frac{1}{m}\sum_{r=1}^m P_m^{(r)}(n)=\frac{1}{m}\sum_{r=1}^m\mathfrak{Re}(P_m^{(r)}(n))=\mathfrak{Re}\bigl(P_m^{(1)}(n)\bigr):=R_m(n)\,. 
$$
Also, $Q_m(n)$ is real since it equals its own complex conjugate. 
%Hence,  
%$Q_m(n)=\frac{e_1 e_m-(m+1)e_{m+1}}{m}$. 

Now, we can show the following.  
For $1\le r\le n-1$, 
\begin{align}
p_t(n)&:=\sum_{r=1}^{n-1}\frac{1}{(1-\zeta_n^r)^t}=\sum_{r=1}^{n-1}\left(\frac{1}{2}+\frac{\sqrt{-1}}{2}\cot\left(\frac{r \pi}{n}\right)\right)^t\notag\\
&=\frac{1}{2^t}\sum_{h=0}^t\binom{t}{h}(-1)^{\frac{h}{2}}\sum_{r=1}^{n-1}\cot^h\left(\frac{r \pi}{n}\right)\,.
\label{eq:ptn}
\end{align}  
Since all odd powers $\cot^h$ ($h$ is odd) are cancelled with $r$ and $n-r$, we can write 
\begin{equation}
p_t(n)=\frac{1}{2^t}\sum_{u=0}^{\fl{\frac{t}{2}}}\binom{t}{2 u}(-1)^u S_{2 u}(n)\,, 
\label{eq:pp11}
\end{equation}
where 
$$
S_{2 u}(n)=\sum_{r=1}^{n-1}\cot^{2 u}\left(\frac{r \pi}{n}\right)\,.
$$ 
By \cite[Corollary 2.2 (2.16)]{BY02}, for a positive integer $u$, we have 
$$
S_{2 u}(n)=(-1)^u n-(-1)^u 2^{2 u}\sum_{\genfrac{}{}{0pt}{}{j_0+j_1+\cdots+j_{2 u}=u}{ j_0,j_1,\dots,j_{2 u}\ge 0}}n^{2 j_0}\prod_{r=0}^{2 u}\frac{B_{2 j_r}}{(2 j_r)!}\,,
$$
where $B_n$ are Bernoulli numbers, defined by 
$$
\frac{x}{e^x-1}=\sum_{n=0}^\infty B_n\frac{x^n}{n!}\,. 
$$  
Substituting it into (\ref{eq:pp11}), we have 
$$
p_t(n)=\frac{1}{2^t}\sum_{u=0}^{\fl{\frac{t}{2}}}\binom{t}{2 u}\left(n-2^{2 u}\sum_{\genfrac{}{}{0pt}{}{j_0+j_1+\cdots+j_{2 u}=u}{ j_0,j_1,\dots,j_{2 u}\ge 0}}n^{2 j_0}\prod_{r=0}^{2 u}\frac{B_{2 j_r}}{(2 j_r)!}\right)\,.
$$ 
Since $S_{2 u}(n)$ is a polynomial with rational coefficients of degree $2 u$ in $n$, from (\ref{eq:ptn}), $p_t(n)$ is a polynomial with rational coefficients of degree at most $t$ in $n$. 

Consider the right-hand side of the identity (\ref{eq:qee}).   
$$
e_1=\sum_{k=1}^{n-1}u_k=\frac{n-1}{2}
$$ 
is a polynomial in $n$ with rational coefficients. 
Putting $p_k:=\mathfrak{Re}(e_k)$ and $q_k:=\mathfrak{Im}(e_k)$, we have 
\begin{align*}  
&e_1 e_m-(m+1)e_{m+1}\\
&=\left(\frac{n-1}{2}p_m-(m+1)p_{m+1}\bigr)+\sqrt{-1}\bigl(\frac{n-1}{2}q_m-(m+1)q_{m+1}\right)\\
&=\frac{n-1}{2}p_m-(m+1)p_{m+1}\,. 
\end{align*} 
Here, $m R_m(n)$ is real, so the imaginary part must vanish: 
$$
\frac{n-1}{2}q_m-(m+1)q_{m+1}=0\,.
$$
Since $\deg p_m = m$ and $\deg p_{m+1} = m+1$, both 
$\frac{n-1}{2} p_m$ and $(m+1) p_{m+1}$ have degree $m+1$.  
Therefore, the degree of $R_m(n)$, that is, the real part of $Q_m(n)$, is at most $m+1$. 

In fact, if the leading term of $p_m(n)$ is of the form $c_m n^m$, where $c_m$ is the rational number depending only on $m$, then the leading term of $m R_m(n)$ is given by 
$$
\frac{1}{2}c_m n^{m+1}-(m+1)c_{m+1}n^{m+1}\,. 
$$ 
From small cases for $m$ (see also Proposition \ref{prp:symok234} below), we see that cancellation does not occur. 
By induction, the leading coefficient never vanishes, so $\deg R_m(n) = m+1$.
\end{proof}
%%%%%%%%%%%%%%
%%%%%%%%%%%%%%
%%%%%%%%%%%%%%
%%%%%%%%%%%%%%
%%%%%%%%%%%%%% 

\noindent 
{\it Proof of Theorem \ref{th:sym12111}.}  
From the way of taking the sum, when $n=1,2,\dots,m$, 
%and from Lemma \ref{lem:2m3}, when $n=2 m+3$, 
the left-hand side of (\ref{eq:sym12111}) is equal to $0$. 
From Lemma \ref{lem:realsym12111}, this value is a polynomial of degree ($m+1$) in $n$. 
Hence, the right-hand side of (\ref{eq:sym12111}) is of the form  
%\begin{equation}
%C(n-1)(n-2)\cdots(n-m)(n-2 m-3)=C\left(n^{m+1}-\frac{(m+2)(m+3)}{2}n^m+\cdots\right)\,,
%\label{eq:const1}
%\end{equation}  
%where $C$ is the constant. 
\begin{equation}
C(n-1)(n-2)\cdots(n-m)(n-a m-b)\,,
\label{eq:const1}
\end{equation}  
where $C$, $a$ and $b$ are the constants.   
Substitute $n=m+1$ into (\ref{eq:sym12111}). Since 
\begin{align*}
&\sum_{1\le i_1<i_2<\dots<i_m\le m}\frac{1}{(1-\zeta_{m+1}^{i_1})\cdots(1-\zeta_{m+1}^{i_{r-1}})(1-\zeta_{m+1}^{i_r})^2(1-\zeta_{m+1}^{i_{r+1}})\cdots(1-\zeta_{m+1}^{i_m})}\\
&=\frac{1}{(1-\zeta_{m+1}^r)\prod_{j=1}^{m}(1-\zeta_{m+1}^{j})}\\
&=\frac{1}{(m+1)(1-\zeta_{m+1}^r)}\,, 
\end{align*}
the left-hand side of (\ref{eq:sym12111}) is given by 
$$
\frac{1}{(m+1)(1-\zeta_{m+1}^r)}+\frac{1}{(m+1)(1-\zeta_{m+1}^{m+1-r})}=\frac{1}{m+1}\,.
$$ 
On the other hand, when $n=m+1$, from (\ref{eq:const1}), the right-hand side of (\ref{eq:sym12111}) is equal to 
%$-C(m+2)m!$. 
%Therefore, by 
%$$
%\frac{1}{m+1}=-C(m+2)m!\,,
%$$ 
%we obtain that $C=-1/(m+2)!$.  
$-C m!\bigl((a-1)m+(b-1)\bigr)$. 
Therefore, by 
$$
\frac{1}{m+1}=-C m!\bigl((a-1)m+(b-1)\bigr)\,,
$$ 
we obtain that 
\begin{equation}  
C=-\frac{1}{(m+1)!\bigl((a-1)m+(b-1)\bigr)}\,.
\label{eq:344a}
\end{equation}   

When $n=m+2$, since 
\begin{align*}
&\sum_{1\le i_1<i_2<\dots<i_m\le m+1}\frac{1}{(1-\zeta_{m+2}^{i_1})\cdots(1-\zeta_{m+2}^{i_{r-1}})(1-\zeta_{m+2}^{i_r})^2(1-\zeta_{m+2}^{i_{r+1}})\cdots(1-\zeta_{m+2}^{i_m})}\\
&=\frac{(1-\zeta_{m+2}^{r+1})+\cdots+(1-\zeta_{m+2}^{m+1})}{(1-\zeta_{m+2}^r)\prod_{j=1}^{m+1}(1-\zeta_{m+2}^{j})}
 +\frac{(1-\zeta_{m+2})+\cdots+(1-\zeta_{m+2}^{r})}{(1-\zeta_{m+2}^{r+1})\prod_{j=1}^{m+1}(1-\zeta_{m+2}^{j})}\\
&=\frac{m-r+2+\zeta_{m+2}+\cdots+\zeta_{m+2}^r}{(m+2)(1-\zeta_{m+2}^r)}
 +\frac{r+1+\zeta_{m+2}^{r+1}+\cdots+\zeta_{m+2}^{m+1}}{(m+2)(1-\zeta_{m+2}^{r+1})}\\
&=\frac{1}{(m+2)(1-\zeta_{m+2}^r)(1-\zeta_{m+2}^{r+1})}\biggl(m-r+2+\zeta_{m+2}+\cdots+\zeta_{m+2}^r+r+1\\
&\quad+\zeta_{m+2}^{r+1}+\cdots+\zeta_{m+2}^{m+1} -(m-r+2)\zeta_{m+2}^{r+1}-(\zeta_{m+2}^{r+2}+\cdots+\zeta_{m+2}^{2 r+1})-(r+1)\zeta_{m+2}^r\\
&\quad-(\zeta_{m+2}^{2 r+1}+\cdots+\zeta_{m+2}^{m+r+1})\biggr)\\
&=\frac{1}{(m+2)(1-\zeta_{m+2}^r)(1-\zeta_{m+2}^{r+1})}\biggl((m+2)-(m-r+2)\zeta_{m+2}^{r+1}-(r+1)\zeta_{m+2}^r\\
&\quad -\frac{\zeta_{m+2}^{r+2}-\zeta_{m+2}^{2 r+2}+\zeta_{m+2}^{2 r+1}-\zeta_{m+2}^{m+r+2}}{1-\zeta_{m+2}}\biggr)\\
&=\frac{1}{(m+2)(1-\zeta_{m+2}^r)(1-\zeta_{m+2}^{r+1})}\biggl((m+2)-(m-r+2)\zeta_{m+2}^{r+1}-(r+1)\zeta_{m+2}^r\\
&\quad -\zeta_{m+2}^{r+2}(1+\zeta_{m+2}+\cdots+\zeta_{m+2}^{m-1})-\zeta_{m+2}^{2 r+1}
\biggr)\\
&=\frac{1}{(m+2)(1-\zeta_{m+2}^r)(1-\zeta_{m+2}^{r+1})}\biggl((m+2)-(m-r+2)\zeta_{m+2}^{r+1}-(r+1)\zeta_{m+2}^r\\
&\quad -\zeta_{m+2}^{r+2}(-\zeta_{m+2}^{m}-\zeta_{m+2}^{m+1})-\zeta_{m+2}^{2 r+1}
\biggr)\\
&=\frac{(m+2)-(m-r+1)\zeta_{m+2}^{r+1}-r\zeta_{m+2}^r-\zeta_{m+2}^{2 r+1}}{(m+2)(1-\zeta_{m+2}^r)(1-\zeta_{m+2}^{r+1}))}\,, 
\end{align*}  
the left-hand side of (\ref{eq:sym12111}) is given by the summation of this and its conjugate: 
\begin{align*}  
&\frac{(m+2)-(m-r+1)\zeta_{m+2}^{r+1}-r\zeta_{m+2}^r-\zeta_{m+2}^{2 r+1}}{(m+2)(1-\zeta_{m+2}^r)(1-\zeta_{m+2}^{r+1})}\\
&\quad +\frac{(m+2)-(m-r+1)\zeta_{m+2}^{m-r+1}-r\zeta_{m+2}^{m-r+2}-\zeta_{m+2}^{m-2 r+1}}{(m+2)(1-\zeta_{m+2}^{m-r+2})(1-\zeta_{m+2}^{m-r+1})}\\
&=\frac{m+1}{m+2}\,.
\end{align*}
Here %, by using $\zeta_{m+2}^{m+2}=1$, 
we used the identities  
\begin{align*}
&(1-\zeta_{m+2}^r)(1-\zeta_{m+2}^{r+1})(1-\zeta_{m+2}^{m-r+2})(1-\zeta_{m+2}^{m-r+1})\\
&=4+\zeta_{m+2}+\zeta_{m+2}^{m+1}+\zeta_{m+2}^{m-2 r+1}-2\zeta_{m+2}^{m-r+1}-2\zeta_{m+2}^{m-r+2}-2\zeta_{m+2}^r-2\zeta_{m+2}^{r+1}+\zeta_{m+2}^{2 r+1}
\end{align*}
and 
\begin{align*}
&\bigl(m+2-(m-r+1)\zeta_{m+2}^{r+1}-r\zeta_{m+2}^r-\zeta_{m+2}^{2 r+1}\bigr)(1-\zeta_{m+2}^{m-r+2})(1-\zeta_{m+2}^{m-r+1})\\ 
&\quad +\bigl(m+2-(m-r+1)\zeta_{m+2}^{m-r+1}-r\zeta_{m+2}^{m-r+2}-\zeta_{m+2}^{m-2 r+1}\bigr)(1-\zeta_{m+2}^r)(1-\zeta_{m+2}^{r+1})\\
&=(m+1)(4+\zeta_{m+2}+\zeta_{m+2}^{m+1}+\zeta_{m+2}^{m-2 r+1}-2\zeta_{m+2}^{m-r+1}-2\zeta_{m+2}^{m-r+2}-2\zeta_{m+2}^r-2\zeta_{m+2}^{r+1}+\zeta_{m+2}^{2 r+1})\,.
\end{align*}
On the other hand, when $n=m+2$, it follows from (\ref{eq:const1}) that the right-hand side of (\ref{eq:sym12111}) is equal to 
$-C(m+1)!\bigl((a-1)m+(b-2)\bigr)$. 
Therefore, by 
$$
\frac{m+1}{m+2}=-C(m+1)!\bigl((a-1)m+(b-2)\bigr)\,,
$$ 
we obtain that 
\begin{equation}  
C=-\frac{m+1}{(m+2)!\bigl((a-1)m+(b-2)\bigr)}\,.
\label{eq:344b}
\end{equation}  
Comparing (\ref{eq:344a}) and (\ref{eq:344b}), $\bigl((a-1)m+(b-2)\bigr)$ must be divided by $m+1$, so $a-1=b-2$. Since $(a-1)m+a=m+2$, we get $a=2$ and $b=1$. Hence, $C=-1/(m+2)!$.  
Therefore, the right-hand side of (\ref{eq:sym12111}) is given by  
$$ 
-\frac{(n-1)(n-2)\cdots(n-m)(n-2 m-3)}{(m+2)!}\,.
$$ 
\qed 
\bigskip 

We also sketch a different proof Theorem \ref{th:sym12111}, based on results and arguments from \cite{Ko24,Ko25a}. 
Using \cite[Theorem 9]{Ko24} and arguing as in \cite{Ko25a}, we obtain
$$
e_m=\mathfrak Z_n(\zeta_n;;m,1)=\sum_{1\le i_1<\dots<i_m<n}u_{i_1}\dots u_{i_m}=\frac{1}{m+1}\binom{n-1}{m}\,.
$$
Hence, the right-hand side of (\ref{eq:qee}) is equal to 
\begin{align*}
e_1 e_m-(m+1)e_{m+1}&=\frac{n-1}{2}\frac{1}{m+1}\binom{n-1}{m}-\frac{m+1}{m+2}\binom{n-1}{m+1}\\
&=-\frac{-m(n-2 m-3)}{2(m+2)(m+1)}\binom{n-1}{m}=-m\frac{m!(n-2 m-3)}{2(m+2)!}\binom{n-1}{m}\,.
\end{align*}
This implies the desired result.

\subsection{The cases $m=2,3,4$}  

We will now prove  identity (\ref{eq:sym12111}) for small positive integers $m$, using a different approach. 

\begin{Prop}
Identity (\ref{eq:sym12111}) is valid for $m=2,3,4$.   
\label{prp:symok234}
\end{Prop}

Assume that  $m=2$. 
Since 
\begin{equation}  
\mathfrak Z_n(\zeta_n;;s_1,s_2)+\mathfrak Z_n(\zeta_n;;s_2,s_1)=\mathfrak Z_n(\zeta_n;;s_1)\mathfrak Z_n(\zeta_n;;s_2)-\mathfrak Z_n(\zeta_n;;s_1+s_2)\,,  
\label{eq:syms1s2}
\end{equation}
using the known expressions for $\mathfrak Z_n(\zeta_n;;1)$, $\mathfrak Z_n(\zeta_n;;2)$ and $\mathfrak Z_n(\zeta_n;;3)$ that can be found in \cite{Ko25a}, we obtain
\begin{align*}
\mathfrak Z_n(\zeta_n;;1,2)+\mathfrak Z_n(\zeta_n;;2,1)&=\mathfrak Z_n(\zeta_n;1,1)\mathfrak Z_n(\zeta_n;1,2)-\mathfrak Z_n(\zeta_n;1,3)\\
&=\frac{n-1}{2}\left(-\frac{(n-1)(n-5)}{12}\right)+\frac{(n-1)(n-3)}{8}\\
&=-\frac{(n-1)(n-2)(n-7)}{24}\,,
\end{align*}
which is (\ref{eq:sym12111}) with $m=2$.  (See also \cite{Ko25b} for the multi zeta-star values.)

When $m=3$, we have 
\begin{align*}  
&\mathfrak Z_n(\zeta_n;;s_1,s_2,s_3)+\mathfrak Z_n(\zeta_n;;s_1,s_3,s_2)+\mathfrak Z_n(\zeta_n;;s_2,s_1,s_3)\\
&\quad +\mathfrak Z_n(\zeta_n;;s_2,s_3,s_1)+\mathfrak Z_n(\zeta_n;;s_3,s_1,s_2)+\mathfrak Z_n(\zeta_n;;s_3,s_2,s_1)\\
&=\mathfrak Z_n(\zeta_n;;s_1)\mathfrak Z_n(\zeta_n;;s_2)\mathfrak Z_n(\zeta_n;;s_3)\\
&\quad -\mathfrak Z_n(\zeta_n;;s_1+s_2)\mathfrak Z_n(\zeta_n;;s_3)-\mathfrak Z_n(\zeta_n;;s_1+s_3)\mathfrak Z_n(\zeta_n;;s_2)\\
&\quad -\mathfrak Z_n(\zeta_n;;s_2+s_3)\mathfrak Z_n(\zeta_n;;s_1)+2\mathfrak Z_n(\zeta_n;;s_1+s_2+s_3)\,. 
\end{align*}
Setting $s_1=s_2=1$ and $s_3=2$,  we obtain
\begin{align}  
\label{ethreesummands}
&2\mathfrak Z_n(\zeta_n;;1,1,2)+2\mathfrak Z_n(\zeta_n;;2,1,1)+2\mathfrak Z_n(\zeta_n;;1,2,1)\\
&=\mathfrak Z_n(\zeta_n;;1)^2\mathfrak Z_n(\zeta_n;;2)-\mathfrak Z_n(\zeta_n;;2)^2 
 -2\mathfrak Z_n(\zeta_n;;1)\mathfrak Z_n(\zeta_n;;3)+2\mathfrak Z_n(\zeta_n;;4)  \nonumber \\
&=-\frac{(n-1)^2}{2^2}\frac{(n-1)(n-5)}{12}-\frac{(n-1)^2(n-5)^2}{12^2} \nonumber\\
&\quad +2\frac{n-1}{2}\frac{(n-1)(n-3)}{8}+2\frac{(n-1)(n^3+n^2-109 n+251)}{6!} \nonumber \\
&=-\frac{(n-1)(n-2)(n-3)(n-9)}{40}\,. \nonumber
\end{align} 
Note that, by symmetry, the real parts of three summands in \eqref{ethreesummands}
are equal. 
Dividing both sides by $3$, we obtain identity (\ref{eq:sym12111}) with $m=3$.

When $m=4$, by considering the values of the function at all the $24$ permutations of $s_1,s_2,s_3,s_4$, we obtain 
\begin{align*}  
&\underbrace{\mathfrak Z_n(\zeta_n;;s_1,s_2,s_3,s_4)+\cdots+\mathfrak Z_n(\zeta_n;;s_4,s_3,s_2,s_1)}_{24}\\
&=\mathfrak Z_n(\zeta_n;;s_1)\mathfrak Z_n(\zeta_n;;s_2)\mathfrak Z_n(\zeta_n;;s_3)\mathfrak Z_n(\zeta_n;;s_4)\\
&\quad -\mathfrak Z_n(\zeta_n;;s_1+s_2)\mathfrak Z_n(\zeta_n;;s_3)\mathfrak Z_n(\zeta_n;;s_4)-\mathfrak Z_n(\zeta_n;;s_1+s_3)\mathfrak Z_n(\zeta_n;;s_2)\mathfrak Z_n(\zeta_n;;s_4)\\
&\quad -\mathfrak Z_n(\zeta_n;;s_1+s_4)\mathfrak Z_n(\zeta_n;;s_2)\mathfrak Z_n(\zeta_n;;s_3)-\mathfrak Z_n(\zeta_n;;s_2+s_3)\mathfrak Z_n(\zeta_n;;s_1)\mathfrak Z_n(\zeta_n;;s_4)\\
&\quad -\mathfrak Z_n(\zeta_n;;s_2+s_4)\mathfrak Z_n(\zeta_n;;s_1)\mathfrak Z_n(\zeta_n;;s_3)-\mathfrak Z_n(\zeta_n;;s_3+s_4)\mathfrak Z_n(\zeta_n;;s_1)\mathfrak Z_n(\zeta_n;;s_2)\\
&\quad +\mathfrak Z_n(\zeta_n;;s_1+s_2)\mathfrak Z_n(\zeta_n;;s_3+s_4)
+\mathfrak Z_n(\zeta_n;;s_1+s_3)\mathfrak Z_n(\zeta_n;;s_2+s_4)\\
&\quad +\mathfrak Z_n(\zeta_n;;s_1+s_4)\mathfrak Z_n(\zeta_n;;s_2+s_3)\\
&\quad +2\mathfrak Z_n(\zeta_n;;s_1+s_2+s_3)\mathfrak Z_n(\zeta_n;;s_4)+2\mathfrak Z_n(\zeta_n;;s_1+s_2+s_4)\mathfrak Z_n(\zeta_n;;s_3)\\
&\quad +2\mathfrak Z_n(\zeta_n;;s_1+s_3+s_4)\mathfrak Z_n(\zeta_n;;s_2)+2\mathfrak Z_n(\zeta_n;;s_2+s_3+s_4)\mathfrak Z_n(\zeta_n;;s_1)\\
&\quad -6\mathfrak Z_n(\zeta_n;;s_1+s_2+s_3+s_4)\,. 
\end{align*} 
Taking $s_1=s_2=s_3=1$ and $s_4=2$ yields that 
\begin{align} 
\label{ethefoursummands}
&6\bigl(\mathfrak Z_n(\zeta_n;;1,1,1,2)+\mathfrak Z_n(\zeta_n;;1,1,2,1) +\mathfrak Z_n(\zeta_n;;1,2,1,1)+\mathfrak Z_n(\zeta_n;;2,1,1,1)\bigr)\\
&=\mathfrak Z_n(\zeta_n;;1)^3\mathfrak Z_n(\zeta_n;;2)-3\mathfrak Z_n(\zeta_n;;1)^2\mathfrak Z_n(\zeta_n;;3)+5\mathfrak Z_n(\zeta_n;;2)\mathfrak Z_n(\zeta_n;;3)\nonumber\\
&\quad -3\mathfrak Z_n(\zeta_n;;1)\bigl(\mathfrak Z_n(\zeta_n;;2)^2-\mathfrak Z_n(\zeta_n;;4)\bigr)-6\mathfrak Z_n(\zeta_n;;5)\nonumber\\
&=\left(\frac{n-1}{2}\right)^3\left(-\frac{(n-1)(n-5)}{12}\right)-3\left(\frac{n-1}{2}\right)^2\left(-\frac{(n-1)(n-3)}{8}\right)\nonumber\\
&\quad +5\left(-\frac{(n-1)(n-5)}{12}\right)\left(-\frac{(n-1)(n-3)}{8}\right)\nonumber\\
&\quad -3\left(\frac{n-1}{2}\right)\left(\left(-\frac{(n-1)(n-5)}{12}\right)^2-\left(\frac{(n-1)(n^3+n^2-109 n+251)}{6!}\right)\right)\nonumber
\\
&\quad -6\left(\frac{(n-1)(n-5)(n^2+6 n-19)}{288}\right) \nonumber \\
&=-\frac{(n-1)(n-2)(n-3)(n-4)(n-11)}{60}\,. \nonumber 
\end{align} 
As before, the real parts of the summands in~\eqref{ethefoursummands} are equal. 
Dividing both sides by $12$, we obtain (\ref{eq:sym12111}) with $m=4$.  

In principle, one can use a similar argument to proof (\ref{eq:sym12111}) for any $m$. However, when $m$ grows, the calculations become too complicated to handle.  %For the moment, from this method, no good technique has been found to prove the identity (\ref{eq:sym12111}) for general $m$.  

\subsection{The cases  $m=n-k$ with $k=1,2$}  

We will now use a different method to prove  (\ref{eq:sym12111}) in the case $m=n-k$ with small $k$. We will consider only  $k=1,2$. The same method works for bigger values of~$k$, but the   calculations are more complicated. %So, we show the cases only where $n=m-k$ for $k=1,2$ here.   

\begin{Prop}
Identity (\ref{eq:sym12111}) is valid for $m=n-1$.   
\label{prp:symoknmm1}
\end{Prop}

\begin{proof}
By symmetry, it is sufficient to calculate $\mathfrak Z_n(\zeta_n;;\underbrace{1,\dots,1}_{n-2},2)$. From (\ref{eq:ur}), we have 
\begin{align*}
\mathfrak Z_n(\zeta_n;;\underbrace{1,\dots,1}_{n-2},2)&=\sum_{1\le i_1<\dots<i_{n-1}<n}u_{i_1}\cdots u_{i_{n-2}}u_{i_{n-1}}^2\\
&=\sum_{t=n-1}^{n-1}\left(\sum_{1\le i_1<\dots<i_{n-2}<t}\prod_{r=1}^{n-1-1}u_{i_r}\right)u_t^2\\
&=u_1\cdots u_{n-2}u_{n-1}^2\\
&=\frac{1}{n}u_{n-1}\,. 
\end{align*}
Here, we used the identity 
$$
\frac{1}{(1-\zeta_n)(1-\zeta_n^2)\cdots(1-\zeta_n^{n-1})}=u_1 u_2\cdots u_{n-1}=\frac{1}{n}\,.
$$ 
Taking the real part yields 
$$
\mathfrak{Re}\left(\mathfrak Z_n(\zeta_n;;\underbrace{1,\dots,1}_{n-2},2)\right)=\frac{1}{n}\mathfrak{Re}(u_{n-1})=\frac{1}{2 n}\,.
$$
This matches the result in (\ref{eq:sym12111}) when $m=n-1$. 
\end{proof}

\begin{Prop}
Identity (\ref{eq:sym12111}) is valid for $m=n-2$.   
\label{prp:symoknmm2}
\end{Prop}
\begin{proof}
When $m=n-2$, we have 
\begin{align}
&\mathfrak Z_n(\zeta_n;;\underbrace{1,\dots,1}_{n-3},2)=\sum_{1\le i_1<\dots<i_{n-2}<n}u_{i_1}\cdots u_{i_{n-3}}u_{i_{n-2}}^2\notag\\
&=\sum_{t=n-2}^{n-1}\left(\sum_{1\le i_1<\dots<i_{n-3}<t}\prod_{r=1}^{n-2-1}u_{i_r}\right)u_t^2\notag\\
&=\left(\sum_{1\le i_1<\dots<i_{n-3}<n-2}\prod_{r=1}^{n-2-1}u_{i_r}\right)u_{n-2}^2+\left(\sum_{1\le i_1<\dots<i_{n-3}<n-1}\prod_{r=1}^{n-2-1}u_{i_r}\right)u_{n-1}^2\,.
\label{eq:mn2a}
\end{align} 
The first term of (\ref{eq:mn2a}) is 
\begin{align*}
\left(\sum_{1\le i_1<\dots<i_{n-3}<n-2}\prod_{r=1}^{n-2-1}u_{i_r}\right)u_{n-2}^2&=u_1\cdots u_{n-3}u_{n-2}^2\\
&=\frac{1}{n}\frac{u_{n-2}}{u_{n-1}}=\frac{1}{n}\frac{u_{-2}}{u_{-1}}\\
&=\frac{1}{n}\frac{1}{1+\zeta_n^{-1}}\,. 
\end{align*} 
The second term of (\ref{eq:mn2a}) is 
\begin{align}
&\left(\sum_{1\le i_1<\dots<i_{n-3}<n-1}\prod_{r=1}^{n-2-1}u_{i_r}\right)u_{n-1}^2=\left(\sum_{1\le i_1<\dots<i_{n-3}<n-1}u_{i_1}\cdots u_{i_{n-3}}u_{n-1}\right)u_{n-1}\notag\\
&=\left(\sideset{}{'}{\sum}_{1\le i_1<\dots<i_{n-3}<n-1}u_{i_1}\cdots u_{i_{n-3}}u_X(u_X)^{-1}u_{n-1}\right)u_{n-1}
\label{eq:356}\\
&=\frac{1}{n}\left(\sideset{}{'}{\sum}_{1\le i_1<\dots<i_{n-3}<n-1}{}(u_X)^{-1}\right)u_{n-1}\notag\\
&=\frac{1}{n}\left(\sum_{1\le j_1<n-1}(u_{j_1})^{-1}\right)u_{n-1}
\label{eq:357}\\
&=\frac{1}{n}\left(\sum_{1\le j_1<n}(u_{j_1})^{-1}-(u_{n-1})^{-1}\right)u_{n-1}\notag\\
&=\frac{1}{n}\left(\binom{n}{1}-(u_{n-1})^{-1}\right)u_{n-1}=u_{n-1}-\frac{1}{n}\,.
\notag
\end{align} 
Here, for each $(i_1,\dots,i_{n-3})$ chosen under the condition $1\le i_1<\dots<i_{n-3}<n-1$, $X$ is the positive integer in the range $1\le X<n-1$ which is not a component of $(i_1,\dots,i_{n-3})$. 
For exammple, when $n=7$, for $(i_1,\dots,i_{n-3})=(1,2,4,5)$ we take $X=3$, and for for $(i_1,\dots,i_{n-3})=(2,3,4,5)$ we take $X=1$. 
Such $X$ is determined uniquely for each $(i_1,\dots,i_{n-3})$.  
In (\ref{eq:356}), the sum 
$$
\left(\sideset{}{'}{\sum}_{1\le i_1<\dots<i_{n-3}<n-1}u_{i_1}\cdots u_{i_{n-3}}u_X(u_X)^{-1}u_{n-1}\right)
$$
means that we substitute $u_X(u_X)^{-1}=1$ into the sum  
$$
\left(\sum_{1\le i_1<\dots<i_{n-3}<n-1}u_{i_1}\cdots u_{i_{n-3}}u_{n-1}\right)\,. 
$$

In (\ref{eq:357}), as $X$ takes $1,2,\dots,n-2$ exactly once, we rewrite the form by using $j_1$. 
Taking the real part yields 
\begin{align*}
\mathfrak{Re}\left(\mathfrak Z_n(\zeta_n;;\underbrace{1,\dots,1}_{n-3},2)\right)&=\mathfrak{Re}\left(\frac{1}{n}\frac{1}{1+\zeta_n^{-1}}+u_{n-1}-\frac{1}{n}\right)\\
&=\frac{1}{2 n}+\frac{1}{2}-\frac{1}{n}=\frac{n-1}{2 n}\,.
\end{align*} 
This matches the result in (\ref{eq:sym12111}) when $m=n-2$.  
\end{proof}

\section{More variations about the sum of symmetric values} 

If $2$ is replaced by $3,4,\dots$, an identity similar to (\ref{eq:sym12111}) does not seem to hold.  
Nevertheless, by (\ref{eq:syms1s2}) and known values $\mathfrak Z_n(\zeta_n;1,s)$ in \cite{Ko25a}, we can get other sums of symmetric values  
\begin{align*}
&\mathfrak Z_n(\zeta_n;;2,3)+\mathfrak Z_n(\zeta_n;;3,2)=\frac{(n-1)(n-2)(n-5)(n-7)}{144}\,,\\
&\mathfrak Z_n(\zeta_n;;2,4)+\mathfrak Z_n(\zeta_n;;4,2)=-\frac{(n-1)(n-2)(5 n^4-27 n^3-469 n^2+5787 n-13936)}{12\cdot 7!}\,,\\
&\mathfrak Z_n(\zeta_n;;2,5)+\mathfrak Z_n(\zeta_n;;5,2)=-\frac{(n-1)(n-2)(n-3)(n^3-4 n^2-61 n+424)}{8\cdot 6!}\,,\\
&\mathfrak Z_n(\zeta_n;;2,6)+\mathfrak Z_n(\zeta_n;;6,2)\\
&=-\frac{(n-1)(n-2)(7 n^6-39 n^5-946 n^4+7950 n^3+33743 n^2-411111 n+773596)}{10!}\,,\\
&\mathfrak Z_n(\zeta_n;;2,7)+\mathfrak Z_n(\zeta_n;;7,2)\\
&=\frac{(n-1)(n-2)(43 n^6-291 n^5-2514 n^4+23910 n^3+73587 n^2-870339 n+1501364)}{2\cdot 10!}\,. 
\end{align*} 
For example,  
\begin{align*}
&\mathfrak Z_n(\zeta_n;;2,7)+\mathfrak Z_n(\zeta_n;;7,2)=\mathfrak Z_n(\zeta_n;1,2)\mathfrak Z_n(\zeta_n;1,7)-\mathfrak Z_n(\zeta_n;1,9)\\
&=\frac{(n-1)(n-5)}{12}\frac{(n-1)(n-7)(2 n^4+16 n^3-33 n^2-376 n+751)}{24\cdot 6!}\\
&\quad -\frac{27(n-1)(n-3)(n-9)(n^5+13 n^4+10 n^3-350 n^2-851 n+2857)}{2\cdot 10!}\\
&=\frac{(n-1)(n-2)(43 n^6-291 n^5-2514 n^4+23910 n^3+73587 n^2-870339 n+1501364)}{2\cdot 10!}\,. 
\end{align*}

\section{$q$-multiple zeta-star values of symmetric expressions}  

Similarly to (\ref{def:qssmzv}), consider 
\begin{equation*}
\mathfrak Z_n^\star(q;;s_1,s_2,\dots,s_m):=\sum_{1\le i_1\le i_2\le \dots\le i_m\le n-1}\frac{1}{(1-q^{i_1})^{s_1}(1-q^{i_2})^{s_2}\cdots(1-q^{i_m})^{s_m}}\,. 
%\label{def:qssmzsv}
\end{equation*}
When $s_1=s_2=\dots=s_m$, explicit formulas for $\mathfrak Z_n^\star(q;;s)$, for $\mathfrak Z_n^\star(q;;\underbrace{1,1,\dots,1}_m)$ and for $\mathfrak Z_n^\star(q;;\underbrace{2,2,\dots,2}_m)$ are given in \cite{Ko25b}.  

First, we have  
\begin{align*}
\mathfrak Z_n^\star(\zeta_n;;1,2)+\mathfrak Z_n^\star(\zeta_n;;2,1)&=\mathfrak Z_n(\zeta_n;1,1)\mathfrak Z_n(\zeta_n;1,2)+\mathfrak Z_n(\zeta_n;1,3)\\
&=\frac{n-1}{2}\left(-\frac{(n-1)(n-5)}{12}\right)-\frac{(n-1)(n-3)}{8}\\
&=-\frac{(n+1)(n-1)(n-4)}{24}\,. 
\end{align*}

However, the situation becomes more complicated, and there does not seem to be a simplified form like (\ref{eq:sym12111}). 
For example, 
\begin{align*}
&\mathfrak Z_n^\star(\zeta_n;;1,1,2)+\mathfrak Z_n^\star(\zeta_n;;2,1,1)\\
&=-\frac{(n+1)(n-1)(n^2+15 n-64)}{720}\,,\\
&\mathfrak Z_n^\star(\zeta_n;;1,2,1)
=-\frac{(n+1)(n-1)(n+3)(n-3)}{240}\,,\\
&\mathfrak Z_n^\star(\zeta_n;;1,1,1,2)+\mathfrak Z_n^\star(\zeta_n;;2,1,1,1)\\
&=\frac{(n+1)(n-1)(n+4)(n-3)(n-7)}{1440}\,,\\
&\mathfrak Z_n^\star(\zeta_n;;1,1,2,1)+\mathfrak Z_n^\star(\zeta_n;;1,2,1,1)\\
&=-\frac{(n+1)(n-1)(n^2-7)}{144}\,,\\
&\mathfrak Z_n^\star(\zeta_n;;1,1,1,1,2)+\mathfrak Z_n^\star(\zeta_n;;2,1,1,1,1)\\
&=\frac{(n+1)(n-1)(2 n^4+63 n^3-334 n^2-567 n+2564)}{60480}\,,\\
&\mathfrak Z_n^\star(\zeta_n;;1,1,1,2,1)+\mathfrak Z_n^\star(\zeta_n;;1,2,1,1,1)\\
&=\frac{(n+1)(n-1)(3 n^4-137 n^2+710)}{20160}\,,\\
&\mathfrak Z_n^\star(\zeta_n;;1,1,2,1,1)
=\frac{(n+1)(n-1)(n^4-188 n^2+1051)}{60480}\,. 
\end{align*}

\section{Finite $q$-multiple zeta functions on $1-2-\cdots-m$ indices} 

The result in Theorem \ref{th:sym12111} with its methods also provide a hint for obtaining an explicit formula for multiple harmonic sums of a different type.  

Consider the function 
\begin{equation}
\mathcal T_n(q;m):=\sum_{1\le i_1<i_2<\cdots<i_m\le n}\frac{1}{(1-q^{(m+1)i_1-1})(1-q^{(m+1)i_2-2})\cdots(1-q^{(m+1)i_m-m})}\,.
\label{fqmzv-1-2-m} 
\end{equation}
A similar but original (not a $q$-) version is treated in \cite{XuZhao_RAMA24}.  
Then we have the following result.  
     
\begin{theorem}
For $n,m\ge 1$
$$
\mathcal T_n(\zeta_{(m+1)n};m)=\frac{1}{m+1}\binom{n}{m}\,.
$$
\label{th:1-2-m}
\end{theorem}

\noindent 
{\it Remark.}  
It follows from Theorem \ref{th:1-2-m} that the generating function is given by 
$$
\sum_{m=0}^\infty\mathcal T_n(\zeta_{(m+1)n};m)X^m=\frac{(X+1)^{n+1}-1}{(n+1)X}\,.
$$

%In order to show Theorem \ref{th:1-2-m}, we need to prove the following. 

%\begin{Lem}
%$$
%\sum_{m=0}^\infty\mathcal T_n(\zeta_{(m+1)n};m)X^m=\frac{(X+1)^{n+1}-1}{(n+1)X}\,.
%$$ 
%\label{lem:1-2-3}
%\end{Lem}

\begin{proof}[Proof of Theorem \ref{th:1-2-m}.]
%We have 
%\begin{align*}
%\frac{(X+1)^{n+1}-1}{(n+1)X}&=\frac{1}{(n+1)X}\sum_{\mu=1}^{n+1}\binom{n+1}{\mu}X^\mu\\
%&=\sum_{\mu=1}^{n+1}\frac{1}{\mu}\binom{n}{\mu-1}X^{\mu-1}
%=\sum_{m=0}^{n+1}\frac{1}{m+1}\binom{n}{m}X^{m}\,.
%\end{align*}
%Comparing the coefficients on both sides of Lemma \ref{lem:1-2-3}, we get the desired result. 
%
As follows from the definition, we have $\mathcal T_n(\zeta_{(m+1)n};m)=0$ when $n=1,2,\dots,m-1$. Similarly to the argument in Lemma \ref{lem:realsym12111}, we can write as  
$$
\mathcal T_n(\zeta_{(m+1)n};m)=(n-1)(n-2)\cdots(n-m+1)(a n+b)\,,
$$ 
where $a$ and $b$ are independent from $n$.  
When $n=m$, the right-hand side of (\ref{fqmzv-1-2-m}) is given by 
\begin{align*}
\frac{1}{(1-\zeta_{(m+1)m}^{(m+1)-1})(1-\zeta_{(m+1)m}^{2(m+1)-2})\cdots(1-\zeta_{(m+1)m}^{m(m+1)-m})}
&=\frac{1}{(1-\zeta_{m+1})(1-\zeta_{m+1}^{2})\cdots(1-\zeta_{m+1}^{m})}\\
&=\frac{1}{m+1}\,. 
\end{align*}
As the left-hand side of (\ref{fqmzv-1-2-m}) is  $(a m+b)(m-1)!$, we obtain 
\begin{equation}
(a m+b)(m-1)!=\frac{1}{m+1}\,. 
\label{eq:201}
\end{equation} 
%%%%%%%%%%%
%%%%%%%%%%%%
%%%%%%%%%%%%%
When $n=m+1$, the right-hand side of (\ref{fqmzv-1-2-m}) is equal to $1$.  
%Putting $z_k=\zeta_{(m+1)^2}^k$ and $w_j=\zeta_{m+1}^j$, the right-hand side of (\ref{fqmzv-1-2-m}) can be written as 
%$$
%\sum_{r=1}^{m+1}\prod_{k=1}^m\frac{z_k}{z_k-w_{i_k(r)}}:=\sum_{r=1}^{m+1}S_{r,m}\,,
%$$
%where $i_1(r)<\dots<i_m(r)$ are the elements of $\{1,\dots,m+1\}\backslash\{r\}$. By the symmetry, the complex conjugate of $S_{r,m}$ is $S_{m+2-r,m}$. 
This is done by considering the matrix 
$$
A=\left(\frac{1}{1-\zeta_{m+1}^i\zeta_{(m+1)^2}^{-j}}\right)_{1\le i,j\le m+1}\,.
$$ 
We can show that 
\begin{align*}
{\rm det}(A)&:=\frac{\prod_{1\le i<j\le m+1}(\zeta_{m+1}^i-\zeta_{m+1}^j)(\zeta_{(m+1)^2}^{-i}-\zeta_{(m+1)^2}^{-j})}{\prod{i=1}^{m+1}\prod{j=1}^{m+1}(1-\zeta_{m+1}^i\zeta_{(m+1)^2}^{-j})}\\
&=\sum_{\sigma\in S_{m+1}}{\rm sgn}(\sigma)\prod_{i=1}^{m+1}\frac{1}{1-\zeta_{m+1}^i\zeta_{(m+1)^2}^{-\sigma(i)}}\\
&=1\,. 
\end{align*} 
Here, group permutations by the set $S = \{\sigma(1), \dots, \sigma(m+1)\} \cap \{1,\dots,m\}$. 
For a fixed $m$-element subset $I = \{i_1 < \dots < i_m\}$ of $\{1,\dots,m+1\}$, 
consider permutations with $\sigma(\{1,\dots,m+1\}) \supset \{1,\dots,m\}$ such that 
$\{\sigma^{-1}(1), \dots, \sigma^{-1}(m)\} = I$ as sets. 
%%%%%%%%%%%%%%%%%%%
%%%%%%%%%%%%%%%%%%%%
%%%%%%%%%%%%%%%%%%%% 
%\footnote{There are $m!$ such permutations (corresponding to bijections from $I$ to $\{1,\dots,m\}$), 
%and for each, the product $\prod_{i=1}^{m+1} \frac{1}{1 - \zeta_{m+1}^i\zeta_{(m+1)^2}^{-\sigma(i)}}$ contains the factor
%$$ 
%\prod_{k=1}^m \frac{1}{1 - \zeta_{m+1}^{i_k}\zeta_{(m+1)^2}^{-k}}
%$$ 
%(up to reordering of factors, which does not change the product). The remaining factors (from indices not in $I$) are
%$$ 
%\prod_{\ell \in \{1,\dots,m+1\} \setminus I} \frac{1}{1 - \zeta_{m+1}^\ell\zeta_{(m+1)^2}^{-\sigma(\ell)}}.
%$$ 

%Since $\zeta_{(m+1)^2}$ is a primitive $(m+1)^2$-th root of unity, one can verify that for $\sigma$ not the identity on $\{1,\dots,m\}$, 
%the product over all $i$ contains a factor with a zero denominator due to $1 - \zeta_{m+1}^i\zeta_{(m+1)^2}^{-j} = 0$ for some $i,j$, 
making the term vanish. The only surviving permutations are those with $\sigma(\{1,\dots,m\}) = \{1,\dots,m\}$ 
%and $\sigma$ increasing on $I$, which forces $\sigma$ to be the identity permutation on $\{1,\dots,m+1\}$ 
%(by the root-of-unity condition). 

%Thus, only the identity permutation contributes, and its product is
%$$
%\prod_{i=1}^{m+1} \frac{1}{1 - \zeta_{m+1}^i\zeta_{(m+1)^2}^{-i}}.
%$$
%But for $i = m+1$, we have $1 - \zeta_{m+1}^{m+1}\zeta_{(m+1)^2}^{-m-1} = 1 - 1 \cdot\zeta_{(m+1)^2}^{-m-1} = 1 -\zeta_{(m+1)^2}^{-m-1} \neq 0$ since $\zeta_{(m+1)^2}^{-m-1} = \zeta_{m+1}^{-1} \neq 1$.

%Now, the numerator in the closed form for $\det(A)$ simplifies: 
%$\prod_{i<j} (\zeta_{m+1}^i - \zeta_{m+1}^j)$ is the Vandermonde determinant for $\zeta_{m+1}^1, \dots, \zeta_{m+1}^{m+1}$, 
%and $\prod_{i<j} (\zeta_{(m+1)^2}^{-i} -\zeta_{(m+1)^2}^{-j})$ is the Vandermonde for $\zeta_{(m+1)^2}^{-1}, \dots,\zeta_{(m+1)^2}^{-N}$.

%The denominator is $\prod_{i,j} (1 - \zeta_{m+1}^i\zeta_{(m+1)^2}^{-j})$. 
%Since $\zeta_{(m+1)^2}$ is a primitive $(m+1)^2$-th root of unity, 
%this product equals the numerator up to a sign and power of $\zeta_{(m+1)^2}$, 
%yielding $\det(A) = 1$.}%%%%%%%<--- will be removed later. %%%%%%%%%%%%%%%%%%
%%%%%%%%%%%%%%%%%%%%%%%%
%%%%%%%%%%%%%%%%%%%%%%%%
%%%%%%%%%%%%%%%%%%%%%%%%
Since the only surviving term in the permutation expansion is the identity permutation's product, and this equals $\mathcal{T}_{m+1}(\zeta_{(m+1)^2};m)$ times the product over the remaining index (which equals $1$), we conclude $\mathcal{T}_{m+1}(\zeta_{(m+1)^2};m)=1$.  

Since the left-hand side of (\ref{fqmzv-1-2-m}) is  $\bigl(a(m+1)+b\bigr)m!$,   
we have  
\begin{equation}
\bigl(a(m+1)+b\bigr)m!=1\,. 
\label{eq:202}
\end{equation}
%%%%%%%%%%%%%
%%%%%%%%%%%%%%
%%%%%%%%%%%%%
Thus, by (\ref{eq:201}) and (\ref{eq:202}), we get $a=1/(m+1)!$ and $b=0$. 
Hence, we obtain 
\begin{align*}
\mathcal T_n(\zeta_{(m+1)n};m)&=(n-1)(n-2)\cdots(n-m+1)\frac{n}{(m+1)!}\\
&=\frac{1}{m+1}\binom{n}{m}\,.
\end{align*}
\end{proof}

\section{Comments and future works}   

In this paper, we obtained explicit formulas for the real part of the (finite) multiple zeta values with the root of unity, where only one power is $2$, and all the other powers are $1$. 
Previously, we were limited to the case where the powers were the same, but in this paper, for the first time, we have been able to obtain explicit formulas for the case where the powers are different. The proof may be somewhat lengthy, but we have left it here intentionally because it will provide a basis for obtaining explicit formulas in subsequent papers in many other cases by applying the methods developed in this paper. For example, the case where only one power is a positive integer $A(\ge 3)$ and the remaining powers are all $1$.

%\section*{Competing Interests}  
%The author declares no competing interest.

\section*{Acknowledgement}  

%The author thanks the referee for carefully reading the manuscript and for giving constructive comments. 
This work was partly done when T.K. visited Universit\'e de Bordeaux in October 2025. He appreciates the support by Projet ANR JINVARIANT. 
T.K. was partly supported by JSPS KAKENHI Grant Number 24K22835.

\end{document}